\documentclass[12pt]{amsart}
\usepackage{amssymb,amsthm, amsmath}
\usepackage{eucal,mathrsfs}
\usepackage{color}
\usepackage[all]{xy}

\newtheoremstyle{mythm}%
  {\topskip}
  {\topskip}
  {\itshape}
  {}
  {\bfseries}
  {\\}
  {\parindent}
  {\thmname{#1}\thmnumber{ #2}\thmnote{ \textit{#3}}}
\newtheoremstyle{mydef}%
  {\topskip}
  {\topskip}
  {}
  {}
  {\bfseries}
  {\\}
  {\parindent}
  {\thmname{#1}\thmnumber{ #2}\thmnote{ #3}}

\theoremstyle{mythm}
\newtheorem{lem}{Lemma}[section]
\newtheorem{thm}[lem]{Theorem}
\newtheorem*{thm*}{Theorem}

\newtheorem{cor}[lem]{Corollary}
\newtheorem*{cor*}{Corollary}
\newtheorem{prop}[lem]{Proposition}

\newtheorem*{conj*}{Conjecture}

\theoremstyle{mydef}
\newtheorem{dfn}[lem]{Definition}
\newtheorem{rem}[lem]{Remark}

\newcommand{\mbb}[1]{\mathbb #1}

\newcommand{\mc}[1]{\mathcal #1}

\newcommand{\oper}[1]{\operatorname{#1}}

\newcommand{\hra}{\hookrightarrow}

\newcommand{\Gm}{\mbb G_m}
\newcommand{\mult}{^{\text{\raisebox{.1ex}{$\times$}}}}

\newcommand{\Br}{\oper{Br}}
\newcommand{\per}{\oper{per}}
\newcommand{\res}{\oper{res}}
\newcommand{\ind}{\oper{ind}}
\newcommand{\lcm}{\oper{lcm}}
\renewcommand{\deg}{\oper{deg}}
\newcommand{\ram}{\oper{ram}}
\newcommand{\Gal}{\oper{Gal}}

\newcommand{\Hom}{\oper{Hom}}

\newcommand{\cha}{\oper{char}}
\newcommand{\cross}[3]{\Delta\!\left(#1, #2, #3\right)}
\newcommand{\Spec}{\oper{Spec}}

\newcommand{\Mat}{\oper{Mat}}

\newcommand{\s}{{\oper{s}}} 
\newcommand{\Ind}{\operatorname{Ind}}
\renewcommand{\P}{{\mathbb P}}
\newcommand{\PP}{\operatorname{PP}}

\newcommand{\Q}{\mbb Q}
\newcommand{\Z}{\mbb Z}
\newcommand{\N}{\mbb N}
\newcommand{\C}{\mbb C}
\renewcommand{\P}{\mbb P}

\newcommand{\til}{\widetilde}
\newcommand{\wh}{\widehat}

\usepackage[OT2,T1]{fontenc}

\def\<{\left<}
\def\>{\right>}

\title{Patching subfields of division algebras}

\author{David Harbater}
\author{Julia Hartmann}
\author{Daniel Krashen}

\date{Version of October 13, 2009.  
The authors were respectively supported in part by NSF Grant DMS-0500118, 
the German National Science Foundation (DFG), and an NSA Young Investigator's Grant.
}

\begin{document}

\maketitle

\begin{abstract}
Given a field $F$, one may ask which finite groups are Galois groups of field extensions $E/F$ such that $E$
is a maximal subfield of a division algebra with center $F$. This question was originally posed by Schacher,
who gave partial results in the case $F = \mathbb Q$. Using patching, we give a complete characterization
of such groups in the case that $F$ is the function field of a curve over a complete discretely valued field with
algebraically closed residue field of characteristic zero, as well as results in related cases.
\end{abstract}

\section{Introduction}

In this manuscript we consider a problem, posed by Schacher in \cite{Sch:subfields}, that relates inverse Galois theory to division algebras.  Given a field $F$, Schacher asked which finite groups $G$ are \textit{admissible} over~$F$, meaning that there is a $G$-Galois field extension $E/F$ with the property that $E$ is a maximal subfield of an $F$-division algebra $D$.  Like the original inverse Galois problem, this problem is generally open; but unlike the original problem, the set of groups that can arise in this manner is often known to be quite restricted (even for $F = \Q$).

This problem is a natural one because of the relationship between maximal subfields of division algebras and crossed product algebras over a given field.  Crossed product algebras can be described explicitly, and are well understood.  In those terms, the above problem can be rephrased as asking for the set of groups $G$ for which there exists an $F$-division algebra that is a crossed product with respect to $G$.

Past work on admissibility has concentrated on the case of global fields.
In \cite{Sch:subfields}, Schacher gave a criterion that is necessary for admissibility of a group over the field $\mbb Q$, and which he
conjectured is also sufficient:

\begin{conj*}[\cite{Sch:subfields}]
Let $G$ be a finite group. Then $G$ is admissible over $\mbb Q$ if and
only if every Sylow subgroup of $G$ is metacyclic.
\end{conj*}

Although still open in general, many particular groups and types of groups
satisfying this criterion have been shown in fact to be admissible over $\mbb Q$; see
for example
\cite{Sonn:solv,SchSonn:A6-A7,ChSonn,FeitFeit,FeiVoj,Feit:2A6-2A7,Feit:sl2-11Q,Feit:psl2-11}.
  Also, Corollary~10.3 of \cite{Sch:subfields} shows that admissible
groups over a global field of characteristic $p$ have metacyclic Sylow
subgroups at the primes other than $p$.

The main theorem of our paper is the following result (see
Theorem~\ref{main_thm}):

\begin{thm*}
Let $K$ be a field that is complete with respect to a discrete
valuation and whose residue field $k$ is algebraically closed.
Let $F$ be a finitely generated field extension of $K$ of transcendence degree one.  Then a finite
group $G$ with $\cha(k) \nmid |G|$ is admissible over $F$ if and only if
every Sylow subgroup of $G$ is abelian of rank at most $2$.
\end{thm*}

The forward direction of this theorem (Proposition~\ref{main_forward})
is analogous to Schacher's results.  As in \cite{Sch:subfields}, a
key ingredient is the equality of period and index in the Brauer group.
The converse direction to our theorem is proven using patching methods from
\cite{HH:FP}, an approach that is not available in the case of global fields, and which makes possible a variety of results for function fields as in the above theorem.  In fact, the equality of period and index used in the forward direction can also be proven by such methods (see \cite{HHK:uinv}).  

In the equal characteristic zero situation, the base field $K$ is
quasi-finite  (i.e.\ perfect with absolute Galois group $\widehat \Z$;
see \cite{Serre:LF}, XIII.2), and $F$ is thus analogous to a global
field, viz.\ to a function field over a finite field.  In that situation, our main theorem
provides a necessary and sufficient condition for an arbitrary finite
group to be admissible over $F$.

This manuscript is organized as follows.  Section~\ref{generalities}
provides background and introduces the notion of an element of the
Brauer group being ``determined by ramification''.  This notion is used
in Section~\ref{sec_admis} to obtain a criterion
(Theorem~\ref{ram_main_thm}) that is then applied to prove the forward
direction of our main theorem.  That section concludes with two
corollaries on admissibility for rational function fields.  Finally,
Section~\ref{sec_patch} recalls ideas from \cite{HH:FP} concerning
patching, and uses them to prove the converse direction of our main
result, Theorem~\ref{main_thm}.

\section*{Acknowledgments}

The authors thank Max Lieblich, Danny Neftin, David Saltman, and Jack Sonn
for helpful discussions during the preparation of this manuscript. We also extend our thanks for the especially valuable comments that we received from the anonymous referee.
\section{Brauer Groups and Ramification}\label{generalities}

The notions of admissibility and crossed product algebras can be understood in terms of Brauer groups of fields, and we review that relationship in this section.  In the case of discretely valued fields, we also describe properties of the ramification map, which associates a cyclic field extension of the residue field to each Brauer class.  This ramification map is an important tool in studying the Brauer group, and 
Proposition~\ref{determined_by_ram_prop} below will play a key role in the next section.

First, we recall some standard facts about central simple algebras and
(central) division algebras; for more detail see \cite{Pie}, \cite{Sa:LN}, \cite{Jac:DA} or \cite{GiSz}.  The \textit{degree} of a central simple $F$-algebra $A$ is the square root of its $F$-dimension, and its \textit{(Schur) index} is the degree of the division algebra $D$ such that $A \cong \Mat_r(D)$ for some $r \geq 1$.  Equivalently, the index is the degree of a minimal \textit{splitting field} for $A$, i.e.\ a field extension $E/F$ such that $A$ splits over $E$ in the sense that $A\otimes_F E$ is a matrix algebra over $F$.  In fact the index divides the degree of any splitting field (\cite{Pie}, Lemma~13.4).  The \textit{Brauer group}  $\Br(F)$ of $F$ consists of the \textit{Brauer
equivalence classes} of central simple $F$-algebras, with the operation of tensor product.
Here two algebras are declared equivalent if they have isomorphic underlying division algebras.  
The order of the class $[A]$ of $A$ in $\Br(F)$ is called its \textit{period} (or \textit{exponent}), and $\per A\, |\, \ind A$ (\cite{Pie}, Proposition~14.4b(ii)).
We will also write $\per(a)$ for the order of an element $a$ in an arbitrary abelian group.

The Brauer group of $F$ may be explicitly identified with the second Galois
cohomology group $H^2(F, \Gm)$.
Concretely, if $E/F$ is a $G$-Galois extension and $c: G \times G \to
E\mult$ is a $2$-cocycle with respect to the standard action of $G$ on
$E\mult$, there is an associated central simple $F$-algebra
$\cross{E}{G}{c}$.  As an $E$-vector space, this algebra has a basis
$u_\sigma$ in bijection with the elements $\sigma$ of the group
$G$.  Multiplication in this algebra is given
by the formulas
\[u_\sigma u_\tau := c(\sigma, \tau) u_{\sigma \tau} \ \ \ \ \ \ \ \ \ \
u_\sigma x := \sigma(x) u_\sigma\]
for $\sigma, \tau \in G$, $x \in E$. We say that an $F$-algebra is a \textit{crossed
product} if it is isomorphic to an algebra of this form. This
construction gives rise (\cite{Sa:LN}, Corollary~7.8) to an isomorphism
\[H^2(F, \Gm) \to \Br(F),\]
implying in particular that every central simple algebra is Brauer
equivalent to a crossed product algebra.

If $A$ is a central simple $F$-algebra of degree $n$, then a commutative 
separable $F$-subalgebra of $A$ is maximal among all such subalgebras if and only if its dimension over $F$ is $n$ 
(\cite{grothbrauer1}, Proposition~3.2).
In the case that $E$ is a subfield, we will refer to $E$ as a \textit{maximal subfield}, following \cite{Sch:subfields}.  Note
that ``maximal subfield'' therefore means not merely that $E$ is maximal as a
subfield, but in fact that it is maximal as a commutative separable subalgebra.  In particular, if $A =
\Mat_2(\C)$, then $A$ has no maximal subfields, since $\C$ has no proper
algebraic extensions.  In the case that $A$ is a division algebra,
however, any commutative separable subalgebra must actually be a field, and so the two notions of maximality agree.
(In \cite{Pie}, \S13.1,  the term ``strictly maximal'' is used for what we call 
``maximal''; and the term ``maximal'' is used there in the weaker sense.)

If $A = \cross{E}{G}{c}$, then $E$ must be a maximal subfield
of $A$, since $\deg A = [E:F]$.  Conversely, if $E$ is a maximal subfield of 
$A$, and $E$ is a $G$-Galois extension of $F$, then $E$ is a splitting field for $A$ (\cite{Pie}, Theorem~13.3) and hence $A =
\cross{E}{G}{c}$ for some $2$-cocycle $c$ (\cite{Sa:LN}, Corollary~7.3).
Thus $G$ is admissible over $F$ if and only if there is a crossed product $F$-division algebra with respect to $G$.

\smallskip

Now consider a field $F$ together with a discrete valuation $v$, completion $F_v$, and residue field $k_v$ at $v$.  
Let $\Br(F)'$ (resp.\ $ H^1(k_v, \Q/\Z)'$) denote the subgroup of elements whose period is prime to $\cha(k_v)$  
(so in particular, $\Br(F)' = \Br(F)$ etc.\ if $\cha(k_v) = 0$). 
Recall from \cite{Sa:LN}, Chapter~10, that there is 
a \textit{ramification map} $\ram_v: \Br(F)' \to H^1(k_v, \Q/\Z)'$,
which factors through the corresponding map on $\Br(F_v)'$.
More generally, if $\Omega$ is a set of discrete valuations on $F$, then we may consider the intersection of the above subgroups of $\Br(F)$, as $v$ ranges over the elements of $\Omega$.  Below, the choice of $\Omega$ will be clear from the context, and we will simply write $\Br(F)'$ for the intersection.  In that situation, for each $v \in \Omega$ the map $\ram_v: \Br(F)' \to H^1(k_v, \Q/\Z)'$ is defined on this intersection.  

\begin{dfn}
Let $F$ be a field, $\Omega$ a set of discrete valuations on $F$, and $\alpha \in \Br(F)'$.  We say that $\alpha$ is
\textit{determined by ramification} (with respect to $\Omega$) if there is some $v \in \Omega$ such that
\[\per(\alpha) = \per(\ram_v\alpha).\]
\end{dfn}

Note that if $\ram_v$ is injective for some $v \in \Omega$, then every class in $\Br(F)'$ is determined by ramification.  This rather special situation is generalized in Proposition~\ref{determined_by_ram_prop} below, which will play an important role in Section~\ref{sec_admis}.  First we make a definition and introduce some notation.

\begin{dfn}
Let $F$ be a field and $\Omega$ a set of discrete valuations on $F$. We define
the \textit{unramified Brauer group} of $F$ (with respect to $\Omega$), denoted by
$\Br_u(F)'$, to be the kernel of
\[\xymatrix{
\Br(F)' \ar[rr]^-{\prod \ram_v} && {\prod_{v \in \Omega}} H^1(k_v, \mbb Q/\mbb Z)'.}\]
\end{dfn}

For a prime $p$ and an integer $n$, let $n_p$ denote the largest power
of $p$ that divides $n$.  A Brauer class $\alpha \in \Br(F)$ may be 
expressed as $\alpha = \alpha_p + (\alpha-\alpha_p)$ for some unique class
$\alpha_p$ (viz.\ a prime-to-$p$ multiple of $\alpha$) satisfying 
$\per(\alpha_p) = (\per(\alpha))_p$ and
$(\per(\alpha-\alpha_p))_p = 1$. Moreover, this class satisfies
$\ind(\alpha_p)=\ind(\alpha)_p$ and $(\ind(\alpha-\alpha_p))_p = 1$.  
(see \cite{Pie}, \S14.4, Primary Decomposition Theorem and its proof). 
Given a division algebra $D$, let $D_p$ be the division algebra in the class $[D]_p$; i.e.\ $[D]_p=[D_p]$.

\begin{prop} \label{determined_by_ram_prop}
Let $F$ be a field, $\Omega$ a nonempty set of discrete valuations on $F$ with
$\Br_u(F)' = 0$, and suppose $\alpha \in \Br(F)$. Then for every prime $p$ not equal to any residue characteristic of $\Omega$,
$\alpha_p$ is determined by ramification.
\end{prop}

\begin{proof}
By definition, $\alpha_p\in \Br(F)'$ for $p$ as in the statement.
By the hypothesis on the unramified Brauer
group, the map 
\[
\xymatrix{
\Br(F)' \ar[rr]^-{\prod \ram_v} & &
\prod_{v \in \Omega} H^1(k_v, \Q/\Z)'}
\]
is injective. It follows that the image of $\alpha_p$ under the map $\prod\ram_v$
has order equal to $\per(\alpha_p)=\per(\alpha)_p$, a power of $p$.  Since the least common multiple of $p$-powers 
is their maximum, the order of the image is the maximum value of
$\per(\ram_v(\alpha_p))$, for $v \in \Omega$. Hence $\per(\alpha_p) =
\per(\ram_v(\alpha_p))$ for some $v \in \Omega$.  So $\alpha_p$ is determined by
ramification. 
\end{proof}

\section{Admissibility Criteria} \label{sec_admis}
In this section, we give a criterion for a group to be admissible.  It is based on an auxiliary lemma together with
an analysis of the ramification map defined in Section~\ref{generalities}.  
In Theorem~\ref{ram_main_thm}, we give a ramification condition on a crossed product algebra that implies that the Sylow subgroups 
of the associated group are metacyclic, or even abelian of rank at most two. We then describe a situation in which this condition
is met (Proposition~\ref{main_forward}), 
thereby providing the forward direction of the main result, Theorem~\ref{main_thm}. 
Finally, we state some corollaries for (retract) 
rational varieties.

We begin with the auxiliary lemma based on Kummer theory.

\begin{lem} \label{ram_abelian_lem}
Let $E/F$ be a finite Galois extension of complete discretely valued fields,
with Galois group $G$ and with cyclic residue field extension $\ell/k$ of characteristic not dividing $|G|$.
Then:
\renewcommand{\theenumi}{\alph{enumi}}
\begin{enumerate}
\item \label{metacyclic}
The Galois group $G$ is metacyclic.
\item \label{abelrk2}
Let $e$ be the ramification index of $E$ over $F$. Then $G$ is abelian (of
rank at most~2) if and only if $F$ contains a primitive $e$-th root of unity.
\end{enumerate}
\end{lem}

\begin{proof}
(\ref{metacyclic})
Let $E_0$ be the maximal unramified extension of $F$ contained
in $E$.  Thus $E/E_0$ is Galois and totally ramified (\cite{Serre:LF}, Corollary ~III.5.3), 
and it is tamely ramified since $\cha(k) \,{\not |}\, [E:E_0]$. 
Hence its Galois group is equal to its inertia group, which is cyclic (\cite{Serre:LF}, Corollaries~2 and~4 of IV.2). Note that $\ell$ is the residue field of $E_0$.

Since the residue field extension $\ell/k$ is cyclic, Proposition~I.7.20 and Corollary~II.3.4 of \cite{Serre:LF}
imply that the unramified extension
$E_0/F$ is also cyclic. Thus $E/F$ is metacyclic. 

(\ref{abelrk2})
We first show that $E_0$ always contains a primitive $e$-th root of unity.
Since the tamely ramified extension $E/E_0$ is totally ramified, its degree is equal to its ramification index $e$.  By \cite{hasse}, Chapter~16, p.~249, $E/E_0$ is a radical extension $E=E_0(z)$ for some $z\in E$ with $f:=z^e \in E_0$.  The conjugates of $z$ over $E_0$ are just the multiples of $z$ by the primitive $e$-th roots of unity in a fixed algebraic closure of $E$.  But $E/E_0$ is Galois; so these conjugates and hence these roots of unity lie in $E$.  Thus there is a primitive $e$-th root of unity in the residue field of $E$, or equivalently of $E_0$ (since $E/E_0$ is totally ramified).  By Hensel's Lemma, there is a primitive $e$-th root of unity $\zeta$ in $E_0$.

By the first part of the proof, the extensions $E/E_0$ and $E_0/F$ are cyclic.  Let $\sigma$ be a generator of $\Gal(E/E_0)$ for which
$\sigma(z)=\zeta z$, and let $\tau \in G$ be a lift of a generator of  
$\Gal(E_0/F)$.  Thus $G$ is abelian if and only if $\sigma$ and $\tau$ commute; and it suffices to show that this latter condition is equivalent to the assertion that $\zeta \in F$.   

Since $E_0/F$ is Galois, $\<\sigma\>$ is normal, and thus $\tau^{-1}\sigma\tau = \sigma^i$  
for some $i$ prime to $e$ with $1 \le i < e$. Moreover $\tau(z)^e=\tau(f) \in E_0$,
so $\sigma(\tau(z))= \zeta^j \tau(z)$ for some non-negative integer $j<e$. 
Hence $r := \tau(z)/z^j$ is fixed by $\sigma$; so $r \in E_0$. 

Let $v$ denote the discrete valuation on $F$, 
normalized so that the valuation of a uniformizer is $1$. 
We may extend $v$ to a discrete valuation on $E$, taking values in $\frac1e \Z$.
Since $\tau$ maps the maximal
ideal of the valuation ring of $E$ to itself, it follows that $v(\tau(h)) = v(h)$
for $h \in E$. Also, $v$ takes integral values on $E_0$, because $E_0$ is
unramified over $F$. 

In particular, $v(f),v(r) \in \Z$.  
We have $v(f) = v(z^e) = v(\tau(z)^e) = v(z^{je}r^e) = v(f^jr^e) = jv(f) + ev(r)
\in jv(f) + e\Z$.  But $v(f)$ is prime to $e$ since $E/E_0$ is totally
ramified.  So $j=1$, i.e. $\sigma(\tau(z))=\zeta \tau(z)$.  It follows that 
$\tau(\zeta)^i\tau(z) = \tau(\zeta^i z) = \tau \sigma^i(z) = \sigma\tau(z) =\zeta \tau(z)$, where the
third equality uses that $\tau^{-1}\sigma\tau = \sigma^i$. That is, $\tau(\zeta)^i=\zeta$. Thus
$\zeta \in F $ if and only if $i=1$; i.e.\ if and only if $\sigma\tau=\tau \sigma$.
\end{proof}

The next proposition makes it possible to verify the cyclicity of the residue field extension required in the previous lemma, in certain cases when the image of a Brauer class under the ramification map becomes trivial upon an extension $E/F$.  It uses
the following connection between 1-cocycles and cyclic extensions:

Consider a field $k$, with fixed separable closure $k^\s$ and absolute Galois
group $G_k := \Gal(k^\s/k)$.  An element $\psi \in H^1(k, \Q/\Z) = \Hom(G_k, \Q/\Z)$
with $\per(\psi) = n$ defines an $n$-cyclic field extension $k' := \left(k^\s\right)^{\ker \psi}$ of $k$.  Let $\ell$ be an algebraic field
extension of $k$, viewed as a subfield of an algebraic closure of $k$ containing $k^\s$.  Then there is a natural inclusion of $G_\ell$ into $G_k$ and hence a restriction map $\Hom(G_k, \Q/\Z) \to \Hom(G_\ell, \Q/\Z)$. By Galois theory the image of $\psi$ under this map defines the compositum 
$k'\ell/\ell$ as a cyclic extension of $\ell$.

\begin{prop}\label{ram_split_prop}
Let $(E,w)/(F,v)$ be an extension of discretely valued fields and
let $\ell/k$ denote the residue field extension.  
Let $\alpha \in \Br(F)'$ (with respect to $\{v\}$),
and suppose that $\per(\ram_v(\alpha)) = [E:F]$.
Let $\alpha_E \in \Br(E)'$ be the element induced by $\alpha$.
If $\ram_w(\alpha_{E}) =0$ then $\ell$ is cyclic over $k$ and 
$w$ is the only discrete valuation on $E$ extending $v$.
\end{prop}

\begin{proof}
Let $\chi :=  \ram_v(\alpha) \in H^1(k,\Q/\Z)'$ and let $e:=e_{w/v}$
be the ramification index of $w$ over $v$.
By \cite{Sa:LN}, Theorem 10.4, there is a commutative diagram
\[\xymatrix{
\Br(F)' \ar[rr]^-{\ram_v} \ar[d] & & H^1(k, \Q/\Z)' \ar[d]^{e\cdot
\res} \\
\Br(E)' \ar[rr]_-{\ram_w} & & H^1(\ell, \Q/\Z)'
}\]
where the vertical map on the
right is the composition of the restriction map with multiplication by $e$.

Hence $e \res (\chi) =\ram_w(\alpha_E)= 0$.
Let $k'$ be the cyclic extension of $k$ defined by $e\chi$. 
Then since $\res(e\chi)=e\res(\chi)=0$,
$k'\ell/\ell$ is a trivial extension, i.e. $k'\subseteq \ell$.
Moreover, $[k':k]=\per(e\chi)\geq \per(\chi)/e$, and thus
$$[E:F]=\per(\chi)\leq [k':k]\cdot e\leq [\ell:k]\cdot e\leq [E:F].$$
So the chain of inequalities is actually a chain of equalities and hence
$\ell=k'$ is cyclic and $[E:F]=[\ell:k]\cdot e$.  Hence $w$ is the only 
discrete valuation on $E$ lying over $v$.
\end{proof}

Recall the notation $n_p$ and $\alpha_p$ introduced before
Proposition~\ref{determined_by_ram_prop}.

\begin{thm} \label{ram_main_thm}
Let $F$ be a field, $\Omega$ a set of discrete valuations on $F$, and $p$ a
prime unequal to all of the residue characteristics of $\Omega$.   
Let $D$ be a crossed product division algebra over $F$ with respect to a
finite group $G$. 
If $[D]_p \in \Br(F)'$ is determined by ramification
and $(\per D)_p = (\ind D)_p$, then every $p$-Sylow subgroup $P$ of $G$ is
metacyclic.  If moreover 
$F$ contains a primitive $|P|$-th root of unity, then $P$ is abelian of rank
at most $2$. 
\end{thm}

\begin{proof} 
By hypothesis there is a discrete valuation $v \in \Omega$ such that 
$\per([D]_p) = \per(\ram_v[D]_p)$.
Since $D$ is a crossed product with respect to $G$, 
there is a maximal subfield
$L$ of $D$ that is $G$-Galois over $F$.  Moreover $D$ is split by $L$
(\cite{Sa:LN}, Corollary~7.3), i.e.\ $D_L := D \otimes_F L$ is split.

Let $P$ be a $p$-Sylow subgroup of $G$, and let $w_1,\ldots,w_n$ denote the
extensions of $v$ to the fixed field $L^P$. Then $[L^P:F]=\sum\limits_{i=1}^n
e_{w_i/v}f_{w_i/v}$ where $e_{w_i/v}$ (resp. $f_{w_i/v}$) is the ramification
index (resp. residue degree) of $w_i$ over $v$. Since $p$ does not divide
$[L^P:F]$, there exists an index $j$ for which $p\nmid e_{w_j/v}f_{w_j/v}$. Let
  $w:=w_j$, and let $u$ be any discrete valuation on $L$ extending
  $w$.
 
Let $\beta$ be the class $[D \otimes_F L^P]$ in $\Br(L^P)$. 
In order to invoke Proposition~\ref{ram_split_prop}, we 
will show that the following three conditions are satisfied:
\renewcommand{\theenumi}{\roman{enumi}}
\begin{enumerate}
\item \label{first} $\beta \in \Br(L^P)'$ (with respect to $\{w\}$);
\item \label{second} $\per (\ram_w \beta) = |P|$;
\item \label{third} $\ram_u (\beta_{L}) = 0$.
\end{enumerate}

For (\ref{first}), the algebra $D$ is split by $L$; hence 
the class $\beta=[D\otimes_F L^P]$ 
is split by $L$, which is of
degree $|P|$ over $L^P$. Consequently, 
$\per \beta \mid \ind \beta \mid |P|$. Hence $\per \beta$ is a power of $p$,
which is not the residue
characteristic of $w$. 
Thus
$\beta \in \Br(L^P)'$ as we wanted to show.

For (\ref{second}), consider the commutative diagram (\cite{Sa:LN}, Theorem
10.4) 
\[\xymatrix{
\Br(F)' \ar[rr]^-{\ram_v} \ar[d]^\phi & & H^1(k, \Q/\Z)' \ar[d]^{e_{w/v}\cdot
\res} \\
\Br(L^P)' \ar[rr]_-{\ram_w} & & H^1(\ell, \Q/\Z)'
}\]
where $k$ and $\ell$ denote the respective residue fields.
The vertical right hand map in the diagram is the composition of multiplication by $e_{w/v}$ and the restriction map.  Since $e_{w/v}$ is prime to $p$, multiplication by $e_{w/v}$ is injective on $p$-power torsion.  Also, the degree of $\ell/k$ is prime to $p$, and the composition of restriction and corestriction is multiplication by that degree.  So the restriction map is also injective on $p$-power torsion.  Hence the 
vertical right hand map in the above diagram is injective on $p$-power torsion.

Now $[D_p] = [D]_p = r[D]$ for some $r \in \N$ that is prime to $p$, and so $[D_p \otimes_{F} L^P] = r\beta$.
Thus $[D]_p$ is sent to $r\beta$ under the left hand vertical map   $\phi$.  
Since
$\per \beta$ is a power of $p$, so is $\per(\ram_w \beta)$, and hence
\[\per(\ram_w \beta) = \per(\ram_w r\beta)
 =\per(\ram_w(\phi([D]_p))=\per(\ram_v([D]_p)
.\]  
Recall that $D$ is a crossed product with respect to $G$, 
so $\ind(D)_p=|G|_p=|P|$, and this quantity equals $\per(D)_p$ by hypothesis. 
Thus, using again that $[D]_p$ is determined by ramification at $v$, 
$\per(\ram_w \beta) = \per(\ram_v [D]_p) = (\per D)_p =|P|$, showing (\ref{second}).  

Finally, (\ref{third}) is immediate since $\ram_u(\beta_{L})=\ram_u([D
\otimes_F L])=0$ because $D$ splits over $L$.

Now Proposition~\ref{ram_split_prop} (applied to the $P$-Galois extension $L/L^P$) implies that the residue field extension $\ell/k$ 
is cyclic and that $u$ is the unique discrete valuation on $L$ extending $w$ on $L^P$.  Hence the extension $L_u/(L^P)_w$ of the corresponding completions also has Galois group $P$, and the
conclusions of the theorem follow immediately from Lemma~\ref{ram_abelian_lem}.
\end{proof}

\begin{cor} \label{ram_cor}
Let $F$ be a field and let $\Omega$ be a nonempty set of discrete valuations
on $F$ with $\Br_u(F)' = 0$.  Let $p$ be a prime unequal to all of the residue
characteristics of $\Omega$, and assume that $\per \alpha_p = \ind \alpha_p$
for every $\alpha \in \Br(F)$.  If $G$ is admissible over $F$, then every
Sylow $p$-subgroup $P$ is metacyclic.  If moreover $F$ contains a primitive
$|P|$-th root of unity, then $P$ is abelian of rank at most two. 
\end{cor}

\begin{proof}
Since $G$ is admissible over $F$, there is a crossed product division algebra
$D$ over $F$ with respect to $G$.  By Proposition~\ref{determined_by_ram_prop}
and the assumption on $\Br_u(F)'$, it follows that $[D]_p \in \Br(F)'$ is
determined by ramification.  By hypothesis, $(\per [D])_p = \per([D]_p) =
\ind([D]_p) = (\ind [D])_p$.  So the conclusion follows from
Theorem~\ref{ram_main_thm}. 
\end{proof}

Using this corollary, we obtain the following proposition, which provides a necessary condition for admissibility and implies the forward direction of our main result, Theorem~\ref{main_thm} below.
 
\begin{prop} \label{main_forward}
Let $K$ be a complete discretely valued field with algebraically closed 
residue field $k$, and let
$F$ be a finitely generated field extension of $K$ of transcendence degree one. 
Let $G$ be a finite group that is admissible over $F$. 
Then for every $p \ne \cha(k)$, the Sylow $p$-subgroups of $G$ are each abelian of rank at most two.
\end{prop}

\begin{proof}
Let $p \ne\cha(k)$ be a prime that divides the order of $G$, and let $P$ be
a Sylow $p$-subgroup of $G$. Let $T$ denote the valuation ring of $K$, and choose
a regular connected projective $T$-curve $\wh X$ with function field  $F$.  
(Given $F$, such an $\wh X$ always exists by resolution of singularities; cf.\ \cite{Abh} or \cite{Lip}.)
Let $\Omega$ be the set of discrete valuations corresponding
to the codimension~$1$ points of $\wh X$.
Thus $p$ is not a residue characteristic of any valuation $v$ in $\Omega$.
With respect to $\Omega$ we have $\Br_u(F)' = 0$, by results in \cite{COP} and \cite{grothbrauer2}.  (Namely, $\Br(\wh X) = 0$ by \cite{COP}, Corollary~1.10(b) (Corollary~1.9(b) in the preprint), and then $\Br_u(F)' = \Br(\wh X)' = 0$ by Proposition~2.3 of \cite{grothbrauer2}.) 

Since $p \ne \cha(k)$, it follows by Theorem~5.5 of \cite{HHK:uinv} (or by
\cite{Lie:PI}, Theorem~5.3) that $\per \alpha_p = \ind \alpha_p$ for every $\alpha \in \Br(F)$.  Again since $p
\ne \cha(k)$, the algebraically closed residue field $k$ contains a primitive
$|P|$-th root of unity; thus by Hensel's Lemma, so does $K$ and hence $F$.
Thus the hypotheses of Corollary~\ref{ram_cor} hold, and therefore
each Sylow $p$-subgroup of $G$ is abelian of rank at most $2$.
\end{proof}

We conclude this section with two corollaries 
(which are not used in the remainder of the paper). 
They use Theorem~\ref{ram_main_thm} to give an obstruction for a
function field to be rational, or even retract rational (which is more general; see \cite{Sa:LN}, p.77).

\begin{cor} \label{retract_rational}
Let $F$ be a retract rational field extension of an algebraically closed field~$C$.  Suppose
that there is an $F$-division algebra $D$ with $\per D = \ind D$, such that $D$ is a crossed product for a group $G$ with $\cha(C)  \nmid |G|$.  
Then all Sylow subgroups of $G$ are abelian of rank at most $2$.
\end{cor}

\begin{proof}
Let $p$ be a prime that divides the order of $G$ and let $P$ be a Sylow $p$-subgroup.
Let $\Omega$ be the set of discrete valuations on $F$ that are trivial on $C$. 
By Proposition~11.8 of~\cite{Sa:LN},  $\Br_u(F)'$ is isomorphic to $\Br(C)'$, the subgroup of $\Br(C)$ consisting of elements of order not divisible by $\cha(C)$.  (Note that $\Omega=R_F$ in the notation in~\cite{Sa:LN}.)
But $\Br(C)$ and hence  $\Br(C)'$ is trivial, because $C$ is algebraically closed. 
Thus by
Proposition~\ref{determined_by_ram_prop}, $[D]_p \in \Br(F)'$ is determined by ramification. 
Since $C$ is algebraically closed of characteristic unequal to $p$,
it contains a primitive $p^r$-th root of unity for all $r \in \N$.  Hence so does $F$.
But $\per D = \ind D$ by assumption, so Theorem~\ref{ram_main_thm} implies the assertion.  
\end{proof}

\begin{rem}
This should be compared to a result of Saltman (\cite{Sal:MIF}), 
stating that if $G$ is a finite group with at least one Sylow subgroup that is not
abelian of rank at most $2$, then the center of a \textit{generic} crossed product algebra with group $G$ is not rational.
\end{rem}

In the case of a rational function field in two variables over an algebraically closed field, the condition 
on period and index is satisfied by \cite{dejong-per-ind}; so from Corollary~\ref{retract_rational} we obtain the following

\begin{cor}
Suppose $F = C(x, y)$ where $C$ is algebraically closed. If $G$ is
admissible over $F$ with $\cha(C) \nmid |G|$ then every Sylow subgroup of
$G$ is abelian of rank at most $2$.
\end{cor}

\section{Admissibility and patching}\label{sec_patch}

To prove the converse direction of our main theorem, we first recall the
method of patching over fields as introduced in \cite{HH:FP}. Throughout this section, 
$T$ is a complete discrete valuation ring with uniformizer~$t$, fraction field $K$, and residue field $k$. 
We consider a finitely generated field extension $F/K$ of transcendence degree one.  Let $\wh X$ be a regular connected projective $T$-curve with function field $F$ such that the reduced irreducible components of its closed fiber 
$X$ are regular.
(Given $F$, such an $\wh X$ always exists by resolution of singularities; cf.~\cite{Abh} or \cite{Lip}.)
Let $f:\wh X \to\P^1_T$ be a finite morphism such that the inverse image $S$ of $\infty \in
\P^1_k$ contains all the points of $X$ at which distinct irreducible components
meet.  (Such a morphism exists by Proposition~6.6 of \cite{HH:FP}.)
We will call $(\wh X,S)$ a \textit{regular $T$-model} of $F$. 

Following Section~6 of~\cite{HH:FP}, for each point $Q \in S$ as above 
we let $R_Q$ be the local ring of $\wh X$ at $Q$, and we let
$\wh R_Q$ be its completion at the maximal ideal corresponding to the point $Q$.  Also, for each connected component $U$ of $X \smallsetminus S$
we let $R_U$ be the subring of $F$
consisting of the rational functions that are regular at the points of $U$, and we
let $\wh R_U$ denote its $t$-adic completion.  If
$Q \in S$ lies in the closure $\bar U$ of a component $U$, then there is a
unique branch $\wp$ of $X$ at $Q$ lying on $\bar U$ (since $\bar U$ is regular).  
Here $\wp$ is a height
one prime of $\wh R_Q$ that contains~$t$, and we may identify it with the pair
$(U,Q)$.  We write $\wh R_\wp$ for the completion of the discrete valuation ring obtained by localizing $\wh R_Q$ at its prime ideal $\wp$.  Thus $\wh R_Q$ is naturally contained in $\wh R_\wp$.

In the above situation, with $\wp = (U,Q)$, there is also a natural inclusion $\wh R_U \hookrightarrow \wh R_\wp$.  To see this, first observe that the localizations of $R_U$ and of $R_Q$ at the generic point of $\bar U$ are the same; and this localization is naturally contained in the $t$-adically complete ring $\wh R_\wp$.  Thus so is $R_U$ and hence its $t$-adic completion $\wh R_U$.  

The inclusions of $\wh R_U$ and of $\wh R_Q$ into $\wh R_\wp$, for $\wp = (U,Q)$, induce inclusions of the corresponding
fraction fields $F_U$ and $F_Q$ into the fraction field $F_\wp$ of
$\wh R_\wp$.  Let $I$ be the index set consisting of all $U,Q,\wp$ described above.  Via the above inclusions, the
collection of all $F_\xi$, for $\xi \in I$, then forms an inverse system with respect to the ordering given by setting $U \succ \wp$ and $Q \succ \wp$ if $\wp = (U,Q)$.

Under the above hypotheses, suppose that for every field extension $L$ of $F$, we are given a category $\mc A(L)$ of algebraic structures over $L$ (i.e.\ finite dimensional $L$-vector spaces with additional structure, e.g.\ associative $L$-algebras), along with base-change functors $\mc A(L) \to \mc A(L')$ when $L \subseteq L'$.  An $\mc A$-\textit{patching problem} for $(\wh X,S)$ consists of an object $V_\xi$ in $\mc A(F_\xi)$ for each $\xi \in I$, together with isomorphisms $\phi_{U,\wp}:V_U  \otimes_{F_U} F_\wp \to V_\wp$ and $\phi_{Q,\wp}:V_Q \otimes_{F_Q} F_\wp \to V_\wp$ in $\mc A(F_\wp)$.
These patching problems form a category, denoted by $\PP_{\mc A}(\wh X,S)$, and there is a base change functor $\mc A(F) \to \PP_{\mc A}(\wh X,S)$.  (Note that the above definition of patching problem is equivalent to that given in Section~2 of \cite{HH:FP} for vector spaces, since the restriction of $\phi_{U,\wp}$ to $ V_U \subset V_U  \otimes_{F_U} F_\wp$ is an $F_U$-linear map that induces $\phi_{U,\wp}$ upon tensoring, and similarly for $\phi_{Q,\wp}$.)  

If an object $V \in \mc A(F)$ induces a given patching problem up to
isomorphism, 
we will say that $V$ is a \textit{solution} 
to that patching problem, or that it is \textit{obtained by patching} the objects $V_\xi$.  
We similarly speak of obtaining a morphism over $F$ by patching morphisms in $\PP_{\mc A}(\wh X,S)$.

The next result is given by~\cite{HH:FP}, Theorem~7.1(i,v,vi), in the context
of Theorem~6.4 of that paper. 

\begin{thm}
\label{HH}
Let $K$ be a complete discretely valued field with valuation ring $T$, and let $F/K$ be a finitely generated field extension of transcendence degree one.  Let $(\wh X,S)$ be a regular $T$-model of $F$.
For a field extension $L$ of $F$, let $\mc A(L)$ denote any of the following categories:

\begin{enumerate}
\item \label{assoc_patch}
the category of finite dimensional associative $L$-algebras,
\item \label{Galois_patch} the category of $G$-Galois $L$-algebras for some
fixed finite group $G$, with $G$-equivariant morphisms, or 
\item \label{csa_patch} the category of central simple $L$-algebras with algebra
homomorphisms.
\end{enumerate}
Then the base change functor $\mc A(F) \to \PP_{\mc A}(\wh X,S)$ is an
equivalence of categories. In particular, every $\mc A$-patching problem has a unique
solution. 
\end{thm}

Given a finite group $G$, a subgroup $H \subseteq G$, and an 
$H$-Galois field extension $L/F$, there is an induced $G$-Galois $F$-algebra
$E = \Ind_H^G L$ given by a direct sum of copies 
of $L$ indexed by the left cosets of $H$ in $G$; e.g.\ see~\cite{HH:FP}, Section~7.2.  In particular, if $H=1$, then $E$ is 
a split extension of $F$; i.e., $E\cong F^{\oplus |G|}$.

Following Schacher \cite{Sch:subfields}, we say that a finite field extension $L/F$
is {\em adequate} 
if there exists an $F$-central division algebra $D$ such that 
$L$ is isomorphic to a maximal subfield of $D$ as an $F$-algebra. In
particular, if $L/F$ is an adequate Galois extension with group $G$, then $G$
is admissible over $F$. Note that $F$ is adequate over itself.

The following lemma shows that the objects in
Theorem~\ref{HH}(\ref{Galois_patch}) 
and~(\ref{csa_patch}) may in a sense be patched in
pairs. Note that the condition on the indices $(G:H_Q)$ implies that the
subgroups $H_Q$ generate $G$, which would suffice to obtain a $G$-Galois
field extension by patching (cf.~\cite{Ha:GC}, Proposition~2.2, or the proof of~\cite{HH:FP}, Theorem 7.3). In our set-up, however, a stronger
condition than generation is needed, due to the restriction on admissible groups given in Proposition~\ref{main_forward}. 

\begin{lem} \label{main_lem}
Let $G$ be a finite group, and let $F$ and $(\wh X,S)$ be as in Theorem~\ref{HH}.  Suppose that for each $Q \in S$ we are given a subgroup $H_Q \subseteq G$ and an $H_Q$-Galois adequate field extension $L_Q/F_Q$  such that $L_Q \otimes_{F_Q} F_\wp$ is a split extension $F_\wp^{\oplus|H_Q|}$ of $F_\wp$ for each branch $\wp$ at $Q$.  Assume that the 
greatest common divisor of the indices $(G:H_Q)$ is equal to $1$.  Then  there exists an adequate $G$-Galois field extension $E/F$ such that
$E \otimes_F F_Q \cong E_Q := \Ind_{H_Q}^G L_Q$ for all $Q \in S$.
\end{lem}

\begin{proof}
We first describe our setup for patching.
Let $I$ be the set of all indices $U, Q, \wp$, as in the introduction to this section.
If $\xi$ is either a branch $\wp$ of the closed fiber $X$ at a point $Q \in S$
or else a component $U$ of $X \smallsetminus S$, 
let $L_\xi = F_\xi$ and $H_\xi = 1$ and let $E_\xi = \Ind_{H_\xi}^G L_\xi \cong F_\xi^{\oplus |G|}$, a split extension of $F_\xi$.  
Thus for every $\xi \in I$, the $F_\xi$-algebra
$E_\xi = \Ind_{H_\xi}^G L_\xi$ is a 
$G$-Galois $F_\xi$-algebra, in the sense of \cite{DeIn}, Section~III.1.  
Let $n=|G|$ and let $n_\xi=(G:H_\xi)$ for $\xi \in I$.  Since each field
extension $L_\xi/F_\xi$ is adequate, 
we may choose division
algebras $D_\xi$ that contain the fields $L_\xi$ as maximal subfields. 
Thus $D_\xi$ is a crossed product $F_\xi$-algebra with respect to $H_\xi$ and $L_\xi$; and $D_\xi$ is split over $L_\xi$ (\cite{Sa:LN}, Corollary~7.3).  Let
$A_\xi = \Mat_{n_\xi}(D_\xi)$.  In particular, if $\xi$ is a branch $\wp$ or a component $U$ then
$n_\xi=n$,  $D_\xi=F_\xi$, and $A_\xi = \Mat_n(F_\xi)$.  

\smallskip

For every $\xi \in I$, we claim that
$E_\xi$ embeds as a maximal commutative separable $F_\xi$-subalgebra of $A_\xi$. 
To see this, note that since the central simple $L_\xi$-algebra $D_\xi \otimes_{F_\xi} L_\xi$ is split, 
so is the $L_\xi$-algebra $D_\xi \otimes_{F_\xi} E_\xi$, in the sense of being a direct sum of matrix algebras over $L_\xi$.  It then follows from \cite{DeIn} (Theorem II.5.5 together with Proposition~V.1.2 and Corollary~I.1.11) 
that $E_\xi$ is a maximal
commutative separable $F_\xi$-subalgebra of some central simple $F_\xi$-algebra $B_\xi$ that is Brauer
equivalent to $D_\xi$.  Such an algebra is a matrix algebra over $D_\xi$,
necessarily of degree $n$ over $F_\xi$; 
so $A_\xi$ and $B_\xi$ are each matrix algebras of the same degree, and hence are isomorphic.
This shows the existence of an $F_\xi$-algebra embedding
$\iota_\xi:E_\xi\rightarrow A_\xi$, proving the claim. 

\smallskip

The proof now proceeds in three steps.
First, we want to use patching to obtain a $G$-Galois (commutative
separable) $F$-algebra $E$.  
Observe that for a branch $\wp = (U,Q)$, 
we have isomorphisms $L_U \otimes_{F_U} F_\wp \to \Ind_1^{H_U} F_\wp$ and 
$L_Q \otimes_{F_Q} F_\wp \to \Ind_1^{H_Q} F_\wp$; these yield isomorphisms $\phi_{U,\wp} : E_U \otimes_{F_U} F_\wp
\to E_\wp$ and $\phi_{Q,\wp} : E_Q \otimes_{F_Q} F_\wp
\to E_\wp$.
With respect to the various maps $\phi_{U,\wp},\phi_{Q,\wp}$, the $G$-Galois $F_\xi$-algebras $E_\xi$ 
(for $\xi\in I$) may be patched, by
Theorem~\ref{HH}(\ref{Galois_patch}), to obtain a $G$-Galois $F$-algebra $E$.  That is, $E \otimes_F F_\xi \cong E_\xi$ for all $\xi \in I$, 
compatibly with the above isomorphisms.

\smallskip 

Next, we will also patch the algebras $A_\xi$, compatibly with the inclusions
$\iota_\xi$.
Let $\wp = (U,Q)$ be a branch at a point $Q \in S$, and let $\xi=U$ or $Q$.  
Since $L_\xi \otimes_{F_\xi} F_\wp \cong F_\wp^{\oplus|H_\xi|}$, there is an
embedding 
$L_\xi \hra F_\wp$ of $L_\xi$-algebras by projecting onto the direct summand
corresponding to the identity element in $H_\xi$.  With respect to this embedding, we have 
\[D_\xi \otimes_{F_\xi} F_\wp \cong (D_\xi \otimes_{F_\xi} L_\xi) \otimes_{L_\xi} F_\wp 
\cong \Mat_{|H_\xi|}(L_\xi) \otimes_{L_\xi} F_\wp
\cong \Mat_{|H_\xi|}(F_\wp),\] 
since $D_\xi \otimes_{F_\xi} L_\xi$ is a split central simple $L_\xi$-algebra.  
So there is an isomorphism $\til\psi_{\xi,\wp}:A_\xi \otimes_{F_\xi} F_\wp \to
\Mat_n(F_\wp) = A_\wp$ 
of $F_\wp$-algebras.
By \cite{Jac:DA}, Theorem
 2.2.3(2), 
the $F_\wp$-algebra embedding $\tilde \psi_{\xi,\wp} \circ (\iota_\xi \otimes_{F_\xi} F_\wp) \circ 
\phi_{\xi,\wp}^{-1}:E_\wp \to A_\wp$
extends to an inner automorphism $\alpha_{\xi,\wp}$ of $A_\wp$.  
Then $\alpha^{-1}_{\xi,\wp} \tilde \psi_{\xi,\wp} \circ (\iota_\xi  \otimes_{F_\xi} F_\wp) = \phi_{\xi,\wp}:E_\xi \otimes_{F_\xi} F_\wp \to E_\wp \subseteq A_\wp$.  Let $\psi_{\xi,\wp} = \alpha^{-1}_{\xi,\wp} \til\psi_{\xi,\wp}: A_\xi\otimes_{F_\xi}F_\wp\rightarrow  A_\wp$. 
Hence if $\xi$ is a point of $X$ or
a component of $X\smallsetminus S$, there is the following
commutative diagram: 
\[\xymatrix{
A_\xi\otimes_{F_\xi}F_\wp \ar[rr]^-{\psi_{\xi,\wp}}  & & A_\wp
 \\
E_\xi\otimes_{F_\xi}F_\wp \ar[u]^-{\iota_\xi  \otimes_{F_\xi} F_\wp} \ar[rr]_-{\phi_{\xi,\wp}} & &
E_\wp \ar[u]^-{\iota_\wp}
}\]
By Theorem \ref{HH}(\ref{csa_patch}), we may patch the algebras $A_\xi$ (for all $\xi \in I$) to obtain a
central simple $F$-algebra $A$.  Furthermore, because of the compatibility expressed in the above diagram,
the morphisms $\iota_{\xi}$ patch to give a morphism $\iota : E \to A$.
Therefore $E$ is a commutative separable subalgebra of $A$, and its
dimension over $F$ equals the dimension $n$ of each $E_\xi$ over $F_\xi$. Thus $E$
is maximal in $A$.

\smallskip

Finally, we will show that $A$ is a division algebra, which implies that $E$
is a field and hence is an adequate $G$-Galois field extension of~$F$.  For this,
note that 
\[n/n_\xi = |H_\xi| = [L_\xi : F_\xi] = \deg(D_\xi) = \ind(A_\xi) \mid \ind(A),\]
where the third equality holds because $L_\xi$ is a maximal subfield of the crossed product algebra $D_\xi$ (\cite{Sa:LN}, Corollary~7.3),
and where the divisibility follows from
\cite{Pie}, Proposition~13.4(v). 
Hence 
\[\deg(A) = n = \lcm(n/n_\xi,\ \xi \in I) \mid \ind(A) \mid \deg(A),\] 
where the second equality follows from the hypothesis on the indices $n_\xi=(G:H_\xi)$ being relatively prime (even just for $\xi = Q \in S$), using that $n = |G|$.  So $\deg(A) = n = \ind(A)$, and hence 
$A$ is a division algebra. This finishes the proof.
\end{proof}

Consider a two-dimensional regular local ring $R$, say with fraction field $F$, and whose maximal ideal has generators $f,t$.  Let $w:F^\times \to \Z$ be the $f$-adic discrete valuation on $F$, with valuation ring $W = R_{(f)}$.   The ring $R/f R$ is also a discrete valuation ring, with fraction field $\bar W = W/f W$.  
The corresponding discrete valuation $u:\bar W^\times \to \Z$ is equal to the $\bar t$-adic valuation on $\bar W$, where $\bar t \in \bar W$ is the image of $t \in W$ under the canonical map $\eta:W \to \bar W$.
With  $\Z \times \Z$ ordered lexicographically, there is a rank two valuation $u \circ w:F^\times \to \Z \times \Z$ defined by 
\[(u \circ w)(a) = \bigl(w(a),u(\eta(af^{-w(a)}))\bigr).\eqno{(*)}\]
Note that $u \circ w$ is the valuation whose associated place is the composition of the places associated to the valuations $u$ and $w$; this follows from the fact that both have valuation ring $\eta^{-1}(R/f R) = R + f W$.  (See \cite{Bo:CA}, Chapter~VI, Sections~3-4, for a general discussion.) 

\begin{lem}\label{divalg_buildblocks}
Let $R$ be a complete regular local ring whose maximal ideal has generators $f,t$; let $n$ be a positive integer; and assume that the fraction field $F$ of $R$ contains a primitive $n$-th root of unity $\zeta$.  Let $D$ be the associative $F$-algebra generated by elements $Y,Z$ satisfying the relations 
\[Y^n = \frac{f}{f - t}, \ \ \ \ Z^n = \frac{f-t^2}{f-t-t^2}, \ \ \ \  YZ=\zeta ZY.\]
Then $D$ is a division algebra over $F$.
\end{lem}

\begin{proof}
The $F$-algebra $D$ is the symbol algebra over $F$ with respect to the $n$-th root of unity $\zeta$ and the elements $a=f/(f - t)$ and $b=(f-t^2)/(f-t-t^2)$ of $F^\times$; see \cite{Wad:VT}, (0.3).  This is a central simple $F$-algebra of degree $n$. 

Consider the rank two valuation $v := u \circ w:F^\times \to \Gamma := \Z \times \Z$ as above, with respect to the generators $f,t$ of the maximal ideal of~$R$.  Thus $v$ is given by the expression $(*)$ above.  Observe that $v(a) = (1,-1)$ and $v(b) = (0,1)$.  Thus the images of $v(a)$ and $v(b)$ in $\Gamma/n\Gamma$ generate that group, which has order $n^2$.  Since $F$ contains a primitive $n$-th root of unity, it then follows from \cite{Wad:VT}, Example~4.4, that $D$ is a division algebra over $F$. 
\end{proof}

The following result now provides the converse direction of our main theorem (in fact in a more general situation):

\begin{prop}\label{main_converse}
Let $G$ be a finite group whose Sylow subgroups are abelian
of rank at most $2$. 
Let $F$ be a finitely generated field extension of transcendence degree one over a complete discretely valued field $K$, and 
assume that $F$ contains a primitive $|G|$-th root of unity. 
Then $G$ is admissible over~$F$. 
\end{prop}

\begin{proof}
Let $p_1,\dots,p_r$ be the prime numbers dividing the order of $G$.
For each $i$ let $P_i$ be a Sylow $p_i$-subgroup of $G$.  Let $T$ be the valuation ring of $K$. As in the introduction to this section, 
there is a regular $T$-curve $\wh X$ such that the reduced irreducible components of the closed fiber $X$ are regular.  Choose 
distinct closed points $Q_1,\dots,Q_r \in X$ at which $X$ is regular.  Thus $Q_i$ lies on only one (regular) irreducible component of $X$,
and there is a unique branch $\wp_i$ of $X$ at the point $Q_i$.
By Proposition~6.6 of \cite{HH:FP}, there is a finite morphism
$f:\wh X \to\P^1_T$ whose fiber $S = f^{-1}(\infty) \subset X$ contains all the points $Q_i$ and all the points at which two irreducible components of $X$ meet.  

As above, let 
$I$ be the set of all indices $U, Q, \wp$, and consider the associated rings $\wh R_\xi$ and their fraction fields $F_\xi$ for $\xi \in I$.
To prove the result, it is enough to construct $P_i$-Galois adequate field extensions $L_i = L_{Q_i}$ of $F_{Q_i}$, for $i = 1, \ldots, r$, 
such that $L_i \otimes_{F_{Q_i}} F_{\wp_i}$ is a split extension $F_{\wp_i}^{\oplus|P_i|}$ of $F_{\wp_i}$.  Namely, if this is done, then let $H_Q = 1$ and $L_Q=F_Q$ for every point 
$Q \in S$ other than $Q_1,\dots,Q_r$.  Since the indices of the subgroups $H_{Q_i} := P_i$ are relatively prime for $i=1,\dots,r$, it follows that the indices of all the subgroups $H_Q$ (for $Q \in S$) are relatively prime; and Lemma~\ref{main_lem} will then imply the assertion.

So fix $i$.  By hypothesis, $P_i$ is an abelian $p_i$-group of rank at most~2, say $P_i = C_q \times C_{q'}$.  
Since $\wh X$ is regular and since $Q_i$ is a regular point of the closed fiber $X$, the maximal ideal of
the two-dimensional regular local ring $\wh R_i:=\wh R_{Q_i}$ is generated by two elements $f_i$ and $t_i$, where $t_i \in \wh R_i$ defines the reduced closed fiber of $\Spec(\wh R_i)$.  (Thus $t$ is a power of $t_i$ multiplied by a unit in $\wh R_i$.)  
Consider the elements $a = f_i/(f_i-t_i)$ and $b = (f_i-t_i^2)/(f_i-t_i-t_i^2)$
of $F_{Q_i}$, and let $L_i=F_{Q_i}(y,z)$ be the field extension of $F_{Q_i}$ defined by the relations $y^q=a$, $z^{q'}=b$.
We will show first that $L_i$ is a $P_i$-Galois adequate field extension of $F_{Q_i}$ and
then show that $L_i \otimes_{F_{Q_i}} F_{\wp_i}$ is a split extension of $F_{\wp_i}$. 

By the assumption on roots of unity, Kummer theory applies.  Neither $a$ nor $b$ is a $d$-th power in $F_{Q_i}$ for any $d>1$, so the extensions $F_{Q_i}(y)$ and $F_{Q_i}(z)$
of $F_{Q_i}$ are Galois with group $C_q$ and $C_{q'}$, respectively.  Here $F_{Q_i}(y)$
is totally ramified over the prime $(f_i)$ of $\Spec(\wh R_i)$, whereas $F_{Q_i}(z)$ is unramified there.  So the intersection of these two fields equals $F_{Q_i}$, showing that they
are linearly disjoint over $F_{Q_i}$.  Hence
the Galois group of their compositum $L_i/F_{Q_i}$ is in fact $P_i$ as claimed.

Let $\zeta\in F$ be a primitive $qq'$-th root of unity; this exists since $qq' = |P_i|$ divides $|G|$.
Consider the central simple $F_{Q_i}$-algebra $D_i$ generated by elements $Y,Z$ satisfying the relations $Y^{q'}=y$, $Z^{q}=z$, and $YZ=\zeta ZY$.
Thus $Y^{qq'} = f_i/(f_i-t_i)$ and $Z^{qq'} = (f_i-t_i^2)/(f_i-t_i-t_i^2)$. 
By Lemma~\ref{divalg_buildblocks}, $D_i$ is an $F_{Q_i}$-division algebra of degree $qq'$.    
Inductively, one sees that $YZ^q=\zeta^qZ^qY$ and thus $Y^{q'}Z^q=\zeta^{qq'}Z^qY^{q'}=Z^qY^{q'}$.  Consequently, $yz=zy$ in $D_i$; i.e.\ $L_i$ is a subfield of $D_i$.  But the degree of $L_i/F_{Q_i}$ is $qq'$, so $L_i$ is a maximal subfield; i.e., $L_i$ is adequate.

Finally, we show that $L_i \otimes_{F_{Q_i}} F_{\wp_i}$ is split.
The above elements $a$ and $b$ each lie in the discrete valuation ring $\wh R_{\wp_i}$, and in fact each is congruent to $1$ modulo the maximal ideal $(t_i)$ of that ring.  The reductions of $a$ and $b$ modulo $t_i$ are thus $qq'$-th powers in the residue field of $\wh R_{\wp_i}$, and so Hensel's Lemma implies that $a$ and $b$ are each $qq'$-th powers in $\wh R_{\wp_i}$ (using that $\wh R_{\wp_i}$ is complete $t$-adically and hence $t_i$-adically).  
Tensoring over $\wh R_{\wp_i}$ with $F_{\wp_i}$, we have that $L_i \otimes_{F_{Q_i}} F_{\wp_i}$ is a split extension of $F_{\wp_i}$, as desired.
\end{proof}

Combining the above with Proposition~\ref{main_forward} yields our main theorem:

\begin{thm} \label{main_thm}
Let $F$ be a finitely generated field extension of transcendence degree one over a complete discretely valued field $K$ with 
algebraically closed residue field $k$, and let $G$ be a finite group
of order not divisible by $\operatorname{char}(k)$.
Then $G$ is
admissible over $F$
if and only if each of its Sylow subgroups is abelian
of rank at most $2$.
\end{thm}

\begin{proof}
The forward direction is given by Proposition~\ref{main_forward}.  
The converse follows from Proposition~\ref{main_converse}, which applies because $K$ (and hence $F$) contains a primitive $|G|$-th root of unity by Hensel's Lemma, 
since the residue field $k$ is algebraically closed of characteristic not dividing~$|G|$.
\end{proof}


\def\cprime{$'$} \def\cprime{$'$} \def\cprime{$'$} \def\cprime{$'$}
  \def\cftil#1{\ifmmode\setbox7\hbox{$\accent"5E#1$}\else
  \setbox7\hbox{\accent"5E#1}\penalty 10000\relax\fi\raise 1\ht7
  \hbox{\lower1.15ex\hbox to 1\wd7{\hss\accent"7E\hss}}\penalty 10000
  \hskip-1\wd7\penalty 10000\box7}
\providecommand{\bysame}{\leavevmode\hbox to3em{\hrulefill}\thinspace}
\providecommand{\MR}{\relax\ifhmode\unskip\space\fi MR }

\providecommand{\MRhref}[2]{%
  \href{http://www.ams.org/mathscinet-getitem?mr=#1}{#2}
}
\providecommand{\href}[2]{#2}

\bigskip

\noindent Author information:

\medskip

\noindent David Harbater: Department of Mathematics, University of Pennsylvania, Philadelphia, PA 19104-6395, USA\\ email: {\tt harbater@math.upenn.edu}

\medskip

\noindent Julia Hartmann: Lehrstuhl A f\"ur Mathematik, RWTH Aachen University, 52062 Aachen, Germany\\ email:  {\tt 
hartmann@mathA.rwth-aachen.de}

\medskip

\noindent Daniel Krashen: Department of Mathematics, University of Georgia, Athens, GA 30602, USA\\ email: {\tt dkrashen@math.uga.edu}

\end{document}